\documentclass[a4paper, 11pt]{article}
\usepackage{amsmath,amsfonts,amssymb,graphicx}
\usepackage{a4wide}
\usepackage{hyperref}
\usepackage{color}

\usepackage{cases}

\usepackage{tikz}
\usetikzlibrary{arrows.meta}

\usepackage{amsthm}
\newtheorem{theorem}{Theorem}
\newtheorem{lemma}[theorem]{Lemma}
\newtheorem{corollary}[theorem]{Corollary}
\newtheorem{remark}[theorem]{Remark}

\newtheorem{definition}[theorem]{Definition}

\renewcommand{\SS}{\mathbb{S}}

\newcommand{\BB}{\mathbb{B}}
\newcommand{\R}{\mathbb{R}}
\newcommand{\f}{\mathbf{f}}

\newcommand{\grads}{\nabla_\SS}

\newcommand{\eps}{\varepsilon}

\usepackage{authblk}

\title{Global versus regional internal–external potential field separation
}

\author[1]{X. Huang\thanks{xinpeng.huang@csu.edu.cn}}
\author[2]{C. Gerhards\thanks{christian.gerhards@geophysik.tu-freiberg.de}}
\author[1]{Z. Ren\thanks{renzhengyong@csu.edu.cn}}
\affil[1]{Central South University, School of Geosciences and Info-Physics, China}
 \affil[2]{TU Bergakademie Freiberg, Institute of Geophysics and Geoinformatics, Germany}
\date{\today}
\begin{document}
	
\maketitle
	
\begin{abstract}

Internal–external field separation is crucial for many aspects of geomagnetism, aiming at distinguishing contributions of the magnetic field generated within a given observation surface from those generated in the exterior. When data are available on a full spherical observation surface, this separation is a standard, stable, and widely used procedure dating back to Gauss. However, when data are only available in a subdomain of the observation surface (as is the case for aeromagnetic and ground-based surveys), the situation drastically changes. Here we show that, without prior assumptions, an internal–external field separation is not uniquely possible. Given the geophysically reasonable assumption that the exterior sources, e.g., magnetospheric and ionospheric current systems, are located above a source-free spherical shell, we show that a unique separation becomes possible but that it is highly unstable. The results are based on the spherical Hardy--Hodge decomposition and explain the intrinsic difficulties of regional data-based internal--external potential field separation. 
\end{abstract}

\section{Introduction}

Geomagnetic measurements contain contributions from various sources. The separation of internal and external fields is a fundamental problem of both mathematical and geophysical interest. It is frequently used to separate the geomagnetic field into contributions originating from sources within the domain enclosed by the observation surface and from sources in the exterior of the observation surface (see, e.g., \cite{backus1996,baratchartgerhards16,Mayer2006,olsen10b,Plattner2017,sabaka15}). This problem is mathematically closely related to the characterization of silent magnetizations (i.e., magnetizations whose corresponding magnetic fields vanish inside or outside of the source layer), although the geophysical interpretation is slightly different. The latter has been investigated quite intensively, e.g., \cite{BGK,baratchart13,barvilhar19,gerhards16a,gubbins11,verles19,verles18} in the context of inverse magnetization problems. Disentangling mixed-source geomagnetic signals is essential for interpretation, yet it is generally a challenging inverse problem. In this respect, internal-external field separation in geomagnetism is a comparatively tractable exception: it follows from a canonical decomposition of vector fields on the observation surface, the so-called Hardy-Hodge decomposition.

If the observation surface is a sphere and data are available on the full sphere, the internal-external field separation is unique and stable. In fact, it coincides with the expansion in terms of well-known (Gauss) vector spherical harmonics. This expansion in terms of spherical harmonics underpins almost all global geomagnetic field models (e.g., \cite{Alken2021,olsen20,lesur10,emag2,sabaka20}).

In practice, the feasibility of internal--external field separation also depends on the available measurements and observation platforms. Geomagnetic data are commonly obtained from ground-based observatories and magnetotelluric stations, aeromagnetic campaigns, and satellite missions. Satellites provide near-global coverage at the orbital altitude, making Gauss-type global spherical harmonic decompositions directly applicable (see, e.g., \cite{olsen15,Olsen2012}). Ground-based measurements, by contrast, do not provide such global coverage. However, they can provide continuous time series at fixed locations and are particularly valuable for capturing temporal variations and induction responses over a broad frequency range. Furthermore, at satellite altitudes, ionospheric current systems may lie below, within, or close to the observation shell, which introduces additional complexity into the interpretation of the separated field components, especially in studies related to geomagnetic induction sounding (see, e.g., \cite{gray24,kuv12,kuv15,ren25}). Thus, ground-based and near-surface data offer complementary advantages. In recent years, efforts such as continental-scale arrays have substantially improved the spatial resolution of geomagnetic measurements at the Earth's surface over land. However, data coverage remains geographically heterogeneous and incomplete, particularly over oceanic regions. 

This naturally raises the question whether internal--external field separation can be meaningfully adapted to geographically incomplete observations, that is, to geomagnetic measurements available only on a patch $U \subset \SS$ of the observation sphere $\SS$. More specifically, one may ask which properties of the global theory remain valid and which necessarily fail when global separation is restricted to local data.

A convenient language for analyzing the effect of patch restriction is provided by the Hardy-Hodge framework, e.g., from \cite{BGK,gerhuakeg23}. This formulates internal-external field separation in terms of \emph{function spaces} and extends beyond the spherical-harmonic setting. In this framework, one models external-source contributions by the Hardy space $\mathcal  H_+(\SS)$ and internal-source contributions by the Hardy space $\mathcal  H_-(\SS)$. All other contributions are collected in the space of tangential, divergence-free vector fields $\mathcal H_{df}(\SS)$. That is, any vector field in $L^2(\SS)^3$ can be decomposed into these three contributions. For brevity, we restrict attention here to the two Hardy components and omit the tangential divergence-free part. From a geophysical point of view, such contributions may be relevant at satellite altitude, where source currents can cross the orbit, but they are negligible for the regional ground-based and near-surface settings that mainly motivate the present work. From a mathematical point of view, tangential divergence-free vector fields can introduce additional non-uniqueness, but they do not contribute to the instability mechanism of the separation problem (a brief discussion on this is provided in Appendix \ref{app:divfree}). Therefore, in this paper, we solely focus on the internal-external separation problem for \emph{potential} fields.

\paragraph{Problem of interest.} The central question of this paper is the following localized separation problem: given $\mathbf{d}\in (\mathcal  H_+(\SS)\oplus \mathcal  H_-(\SS))|_U\subset L^2(U)^3$ on a subset (which we call \lq\lq patch\lq\lq) $U\subset\SS$, can one determine $\f_+\in\mathcal H_+(\SS)$, $\f_-\in\mathcal H_-(\SS)$ such that
\begin{align}
(\f_+ + \f_-)\big|_{U}=\mathbf{d}\,?\label{eqn:probform}
\end{align}
Although this problem is of particular significance in many geomagnetic applications, and practical approaches based on Slepian or spherical cap harmonic analysis have been discussed (see, e.g., \cite{plattnersimons14b,Plattner2017,Torta2019}), its fundamental ill-posedness has received much less systematic attention. Rather than proposing another reconstruction algorithm, the present paper aims to clarify the structural limitations of localized internal--external separation, in particular with respect to feasibility, uniqueness, and stability. To this end, we combine recent Hardy--Hodge representations of geomagnetic field components (e.g., \cite{BGK,gerhuakeg23}) with quantitative conditional stability estimates for Cauchy problems and harmonic continuation (e.g., \cite{aleron09,bou10,isakov12,ruesal19}).
Our main results show that localized internal--external separation differs fundamentally from its global counterpart. More precisely:

\smallskip
\noindent
(i) \emph{Non-uniqueness on patches.} Without additional assumptions, localized measurements do not determine internal and external contributions uniquely: there exist nontrivial internal and external fields whose sum vanishes on the observation patch.

\smallskip
\noindent
(ii) \emph{Uniqueness under a geophysically motivated altitude/analyticity condition (source-free shell).}
Uniqueness can be restored if one imposes the following condition: external contributions are assumed to be generated only above some fixed altitude, which is equivalent to requiring a source-free shell adjacent to the observation surface.

\smallskip
\noindent
(iii) \emph{Ill-posedness persists.} Even under such an altitude/analyticity condition, the problem remains  ill-posed: arbitrarily small perturbations of the data on the patch may produce large changes in the separated components.

\smallskip
\noindent
(iv) \emph{Conditional stability.} Assuming appropriate source-size and smoothness bounds on the internal and external contributions, we can derive a logarithmic conditional stability estimate. In particular, the shell thickness enters the constants in the logarithmic estimate, reflecting the role of the source-free annulus in the uniqueness and stability mechanism.

\smallskip
The remainder of the paper is organized as follows. Section \ref{sec:hardy} recalls the Hardy-Hodge decomposition on closed Lipschitz surfaces and its spherical specialization. These ingredients provide the general framework used in the later analysis. Section \ref{sec:localized} studies the uniqueness of localized separation on spherical patches, establishing non-uniqueness under general conditions, and uniqueness under the altitude/analyticity condition. Section \ref{sec:instab} is devoted to the corresponding instability and conditional stability results. 

\paragraph{Convention (sources vs.\ domains).}
Throughout, \emph{internal} and \emph{external} refer to the \emph{location of sources} relative to the observation surface: internal sources lie below the surface and external sources above it. 
Accordingly, an external field is harmonic in the \emph{interior} region enclosed by the observation surface, while an internal field is harmonic in the \emph{exterior} region and decays at infinity. This matches the standard Gauss interpretation and avoids ambiguity when switching between physical and potential-theoretic language. A schematic sketch of the considered setup is provided in Fig. \ref{fig:setup}.

\section{Hardy--Hodge decomposition as a general framework for internal--external field separation}\label{sec:hardy}

\begin{figure}
    \centering
    \includegraphics[width=0.65\textwidth]{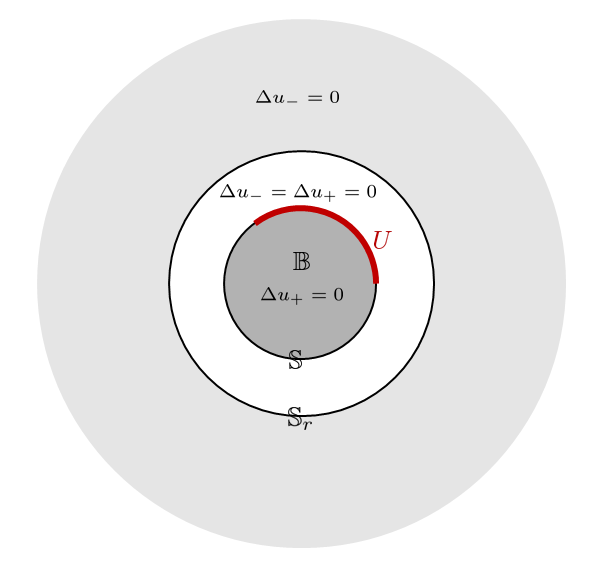}
    \caption{Illustration of the considered setup: data are available on the patch $U\subset\SS$, the internal source contribution on $\SS$ is expressed by $\f_-=\nabla u_-|_\SS$ (with $u_-$ denoting the corresponding harmonic potential in the exterior of the unit ball $\BB$), the external source contribution by $\f_+=\nabla u_+|_\SS$ (with $u_+$ denoting the corresponding harmonic potential in the interior of the  ball $\BB_r$). 
    The radius $r>1$ denotes the separation radius required for the altitude/analyticity condition, i.e., the annulus $A_r=\{x\in\R^3:1<|x|<r\}$ represents a source-free shell. In the limiting case where there is no source-free shell, one may formally set $r=1$, so that the spheres $\SS$ and $\SS_r$ coincide.}
    \label{fig:setup}
\end{figure}

In this section, we recall the Hardy–Hodge decomposition as a functional-analytic framework for global internal--external source separation. For more precise properties of the topics mentioned in this section, we refer the reader to \cite{BGK,gerhuakeg23}. The purpose of this section is mainly preparatory: we summarize the layer-potential characterization of the Hardy--Hodge decomposition on general Lipschitz surfaces, which provides the operator language needed in Sections \ref{sec:localized} and \ref{sec:instab}. We also briefly indicate how, in the spherical setting, this framework recovers the classical (Gauss) vector spherical harmonic separation.

\subsection{Hardy--Hodge decomposition on closed Lipschitz surfaces}\label{sec:HH_Lipschitz}

Let $\Omega\subset\mathbb{R}^3$ be a Lipschitz domain with $\Gamma=\partial\Omega$ its closed boundary, oriented with respect to the outward unit normal
$\eta$. For a vector field $\f\in L^2(\Gamma)^3$ we use the normal--tangential splitting
\begin{equation}\label{eq:nt_split}
\f=\eta f_\eta +\f_T,\qquad f_\eta=\f\cdot \eta,\quad\f_T=\f-(\f\cdot\eta)\eta.
\end{equation}

We recall the boundary integral operators for the Laplacian on $\Gamma$: the boundary single-layer  operator $S$, the boundary double-layer operator $K$, and the tangential gradient operator $\nabla_T S$ together with its adjoint $(\nabla_T S)^\ast$. We refer to \cite{BGK} for precise definitions, mapping properties on Lipschitz surfaces, and the associated trace/jump relations. In what follows, we only use the identities introduced below and the fact that these operators are bounded on the natural $L^2$- and trace spaces.

\paragraph{The $B$-operators and Hardy spaces.}
Following \cite[Defs.~3.2 and 3.15]{BGK}, define the scalar operators $B_+,B_-:L^2(\Gamma)^3\to L^2(\Gamma)$ by
\begin{equation}\label{eq:Bi_Bo}
B_+ \f = -\Big(\tfrac12 - K\Big)f_\eta + (\nabla_T S)^\ast\f_T,
\qquad
B_- \f = \Big(\tfrac12 + K\Big)f_\eta + (\nabla_T S)^\ast\f_T.
\end{equation}
Their adjoints $B_+^\ast,B_-^\ast:L^2(\Gamma)\to L^2(\Gamma)^3$ are given explicitly by
\cite[Lemmas~3.5 and 3.18]{BGK}:
\begin{equation}\label{eq:BiBo_adj}
B_+^\ast g = -\eta\Big(\tfrac12 - K^\ast\Big)g + \nabla_T S\, g,
\qquad
B_-^\ast g = \eta\Big(\tfrac12 + K^\ast\Big)g + \nabla_T S\, g.
\end{equation}
We define the Hardy-type subspaces by
\begin{equation}\label{eq:Hpm_def}
\mathcal H_+ (\Gamma)= \overline{\operatorname{Ran}(B_+^\ast)}^{\,L^2(\Gamma)^3},
\qquad
\mathcal H_- (\Gamma)= \overline{\operatorname{Ran}(B_-^\ast)}^{\,L^2(\Gamma)^3}.
\end{equation}

\paragraph{Potential-theoretic interpretation (physical meaning).}
Although $\mathcal H_\pm(\Gamma)$ are introduced via the boundary operators $B_+^\ast,B_-^\ast$, they admit a direct potential-theoretic characterization. For a Lipschitz surface $\Gamma$, boundary values of $\nabla u$ should be understood as nontangential limits. More precisely, if the inner(+)/outer(-) nontangential maximal functions $(\nabla u)_\pm^M$ are in $ L^2(\Gamma)^3$, then an inner/outer nontangential limit of $\nabla u$ exists a.e. on $\Gamma$ and leads to a well-defined element of $L^2(\Gamma)^3$ (see, e.g., the brief discussion in \cite{BGK,Hunt1968} and references therein). Whenever we write $\nabla u_{(\pm)}|_\Gamma$, it is meant in this nontangential limit sense. Elements of $\mathcal H_+(\Gamma)$ coincide with inner nontangential limits of gradient fields $\nabla u$ generated by harmonic scalar potentials $u$ in the interior domain $\Omega$. Elements of $\mathcal H_-(\Gamma)$ coincide with outer nontangential limits of gradients $\nabla u$ of harmonic potentials $u$ in the exterior domain $\R^3\setminus\overline{\Omega}$ that satisfy a suitable decay condition at infinity. The equivalence between this interpretation and the
operator-based definitions \eqref{eq:Hpm_def} follows from the layer-potential analysis in \cite{BGK}. More precisely:
\begin{align}
	\mathcal H_+(\Gamma)&=\{\nabla u|_\Gamma: \Delta u=0 \textrm{ in }\Omega, \|(\nabla u)_+^M\|_{L^2(\Gamma)^3}<\infty\},
	\\\mathcal H_-(\Gamma)&=\{\nabla u|_\Gamma: \Delta u=0 \textrm{ in }\R^3\setminus\overline{\Omega}, \lim_{x\to\infty} u(x)=0, \|(\nabla u)_-^M\|_{L^2(\Gamma)^3}<\infty\}.
\end{align}
In geomagnetic potential-field applications, it is convenient to read this characterization in terms of source location:
\begin{center}
\begin{tabular}{ll}
external sources (outside of $\Gamma$): & $u$ harmonic in $\Omega$ $\Rightarrow \nabla u|_\Gamma\in \mathcal H_+(\Gamma)$,\\
internal sources (inside of $\Gamma$): & $u$ harmonic in $\R^3\setminus\overline{\Omega}$, $u(x)\to 0$ as $|x|\to\infty$ $\Rightarrow\nabla u|_\Gamma\in \mathcal H_-(\Gamma).$
\end{tabular}
\end{center}

\paragraph{The divergence-free component.}
For the sake of completeness, we introduce the third contribution
\begin{equation}\label{eq:Hdf_def}
\mathcal H_{\mathrm{df}}(\Gamma)
= \Big\{\, \f\in L^2(\Gamma)^3 : \f\cdot\eta = 0 \ \text{and}\ \operatorname{div}_\Gamma \f = 0 \,\Big\},
\end{equation}
where $\operatorname{div}_\Gamma$ denotes the (weak) surface-divergence on $\Gamma$.
This space is $L^2(\Gamma)^3$-orthogonal to $\mathcal H_-(\Gamma) +\mathcal H_+(\Gamma)$ (see \cite{BGK} for the precise functional-analytic formulation).
In the present paper, $\mathcal H_{\mathrm{df}}(\Gamma)$ will not play an explicit role in the localized analysis of
Sections \ref{sec:localized} and \ref{sec:instab}. It is included here only to state the full decomposition of vector fields.

\paragraph{Hardy--Hodge decomposition as a topological direct sum (global stability).}
The main Hardy--Hodge result of \cite{BGK} yields a continuous (topological) direct-sum decomposition of $L^2(\Gamma)^3$ into the above components:
\begin{equation}\label{eq:HHD_topological}
	L^2(\Gamma)^3=\mathcal H_+(\Gamma)+ \mathcal H_-(\Gamma)+ \mathcal H_{\mathrm{df}}(\Gamma).
\end{equation}
That is, every $\f\in L^2(\Gamma)^3$ admits a unique decomposition $\f = \f_+ +\f_- +\f_{\mathrm{df}}$ with $\f_\pm\in\mathcal H_\pm(\Gamma)$ and $\f_{\mathrm{df}}\in\mathcal H_{\mathrm{df}}(\Gamma)$. The mapping $\mathcal{A}:\mathcal H_+(\Gamma) \times \mathcal H_-(\Gamma)\times \mathcal H_{\mathrm{df}}(\Gamma)\to L^2(\Gamma)^3$, $(\f_+,\f_-,\f_{\mathrm{df}})\mapsto \f=\f_++\f_-+\f_{\mathrm{df}}$ is injective. Its inverse  $L^2(\Gamma)^3\to\mathcal H_+(\Gamma) \times \mathcal H_-(\Gamma)\times \mathcal H_{\mathrm{df}}(\Gamma)$, $\f\mapsto(\f_+,\f_-,\f_{\mathrm{df}})$ is continuous (equivalently, the associated projections are bounded). On a general Lipschitz surface this splitting is typically {not} orthogonal, in the sense that $\mathcal H_+(\Gamma)$ and $\mathcal H_-(\Gamma)$ need not be orthogonal to each other, although $\mathcal H_{\mathrm{df}}(\Gamma)$ is orthogonal to $\mathcal H_+ (\Gamma) +\mathcal H_-(\Gamma)$ with respect to the $L^2(\Gamma)^3$-inner product. Nevertheless, when data are available on the full closed surface $\Gamma$, the internal-external separation is \emph{stable} in the sense of continuous dependence of $\f_\pm$ on $\f$. This contrasts sharply with the localized setting studied in Sections \ref{sec:localized} and \ref{sec:instab}. Orthogonality of the function spaces is recovered in the spherical setting, which we review in Section~\ref{subsec:spherical-hhd}.

\subsection{Spherical Hardy--Hodge decomposition and (Gauss) vector spherical harmonics}\label{subsec:spherical-hhd}

We now specify $\Omega=\mathbb{B}=\{x\in\R^3:|x|<1\}$ to be the unit ball and $\Gamma=\SS=\{x\in\mathbb R^3:|x|=1\}$ to be the unit sphere. We write $\eta(x)=x$ for the outward unit normal on $\SS$. Let $\{Y_{n,k}\}_{n\in\mathbb N_0,\;k=-n,\dots,n}$ be an orthonormal basis of scalar spherical harmonics (see, e.g., \cite{freedenschreiner09,mueller96}). We recall the standard vector spherical harmonic families aligned with the spherical Helmholtz decomposition: 
\begin{equation}\label{eq:vsh-families}
\mathbf Y_{n,k}=\eta\,Y_{n,k},\qquad 
\mathbf\Psi_{n,k}=\nabla_\SS Y_{n,k},\qquad
\mathbf\Phi_{n,k}=\eta\times\nabla_\SS Y_{n,k}.
\end{equation}
The fields $\mathbf\Phi_{n,k}$ are tangential and divergence-free (also known as {toroidal} fields), hence they span the tangential divergence-free subspace $\mathcal H_{\mathrm{df}}(\SS)
= \overline{\mathrm{span}\{\mathbf\Phi_{n,k}:n\ge1,\ |k|\le n\}}^{\,L^2(\SS)^3}$.

\paragraph{(Gauss) internal/external poloidal fields.}
We now consider the solid harmonics, with $r=|x|$ and $\xi=\tfrac{x}{|x|}$ for $x\in \mathbb{R}^3\setminus\{0\}$,
\[
V^{\mathrm{ext}}_{n,k}(x)=r^nY_{n,k}(\xi),\quad r<1\qquad \text{and} \qquad
V^{\mathrm{int}}_{n,k}(x)=r^{-(n+1)}Y_{n,k}(\xi), \quad r>1
\]
which represent harmonic functions inside and outside the unit ball $\BB$, respectively.
Using $\nabla = \eta\,\partial_r + r^{-1}\nabla_\SS$, their gradients restricted to $\SS$ are
\begin{equation}\label{eq:gauss-vsh}
\mathbf G^{\mathrm{ext}}_{n,k}=\left.\nabla V^{\mathrm{ext}}_{n,k}\right|_{r\to 1^-}
= \mathbf\Psi_{n,k} + n\mathbf Y_{n,k},
\qquad
\mathbf G^{\mathrm{int}}_{n,k}=\left.\nabla V^{\mathrm{int}}_{n,k}\right|_{r\to 1^+}
= \mathbf\Psi_{n,k}-(n+1)\mathbf Y_{n,k}.
\end{equation}
We follow the {source-location} convention stated in the introduction: {external} and {internal} refer to the location of the sources relative to the observation sphere.
Consequently, external-source fields correspond to potentials that are harmonic {inside} the unit ball (regular solid harmonics), whereas internal-source fields correspond to potentials that are harmonic in the {exterior} of the unit ball (decaying solid harmonics). With this convention, the (Gauss) vector spherical harmonics $\mathbf G^{\mathrm{ext}}_{n,k}$ represent {external-source} contributions and $\mathbf G^{\mathrm{int}}_{n,k}$ represent {internal-source} contributions (see, e.g., \cite{backus1996,freedengerhards12,Mayer2006,olsen10b}).

\paragraph{Connection to the $B$-operators.}
On the sphere, we recall the definition of $B$-operators (note that $K$ is self-adjoint on the sphere)
\begin{equation}\label{eq:Bpm-def}
B_-^{\ast} = \eta\Bigl(K^*+\tfrac12 I\Bigr) + \nabla_\SS S ,\qquad
B_+^{\ast} = \eta\Bigl(K^*-\tfrac12 I\Bigr) + \nabla_\SS S, 
\end{equation}
where $S$ and $K$ denote the (boundary) single- and double-layer operators (notice that, on the sphere, the double layer potential is self-adjoint, i.e., $K^{\ast}=K$ holds). Both layer potential operators diagonalize in the scalar spherical harmonic basis. In particular, one has (see, e.g., \cite{Gerhards2025})
\begin{equation}\label{eq:SK-eigs}
S\,Y_{n,k}=-\tfrac{1}{2n+1}\,Y_{n,k},\qquad
K\,Y_{n,k}=\tfrac{1}{2(2n+1)}\,Y_{n,k}.
\end{equation}
Consequently, for $n\ge1$,
\begin{align}\label{eq:Bpm-on-Y-correct}
B_+^{\ast}(Y_{n,k})
=-\tfrac{1}{2n+1}\Bigl( \mathbf\Psi_{n,k}+n\mathbf Y_{n,k}\Bigr)
=-\tfrac{1}{2n+1}\,\mathbf G^{\mathrm{ext}}_{n,k},
\\
B_-^{\ast}(Y_{n,k})
=-\tfrac{1}{2n+1}\Bigl( \mathbf\Psi_{n,k}-(n+1)\mathbf Y_{n,k}\Bigr)
=-\tfrac{1}{2n+1}\,\mathbf G^{\mathrm{int}}_{n,k}.
\end{align}

\paragraph{Identification of the spherical Hardy spaces.}
Let $\mathcal H_\pm(\SS)$ denote the spherical Hardy spaces (poloidal components) in the sense of the Hardy--Hodge decomposition. Under the above source-location convention, we identify
\begin{align}\label{eq:Hpm-id}
\mathcal H_+(\SS)&=\overline{\mathrm{span}\{\mathbf G^{\mathrm{ext}}_{n,k}: n\ge1,\ |k|\le n\}}^{\,L^2(\SS)^3}
=\overline{\mathrm{Ran}(B_+^{\ast})}^{\,L^2(\SS)^3},
\\
\mathcal H_-(\SS)&=\overline{\mathrm{span}\{\mathbf G^{\mathrm{int}}_{n,k}:n\ge0,\ |k|\le n\}}^{\,L^2(\SS)^3}
=\overline{\mathrm{Ran}(B_-^{\ast})}^{\,L^2(\SS)^3}.
\end{align}
Together with $\mathcal H_{\mathrm{df}}(\SS)$, this yields the spherical Hardy--Hodge decomposition
\begin{equation}\label{eq:spherical-hhd}
L^2(\SS)^3=\mathcal H_+(\SS)\oplus \mathcal H_-(\SS)\oplus \mathcal H_{\mathrm{df}}(\SS),
\end{equation}
and, in contrast to the general Lipschitz-surface setting of Section~\ref{sec:HH_Lipschitz}, the three subspaces in \eqref{eq:spherical-hhd} are mutually orthogonal in $L^2(\SS)^3$.

\begin{remark}[Zero degree case]
For $n=0$ one has $\nabla_\SS Y_{0,0}=0$ and $(K-\tfrac12)Y_{0,0}=0$. Thus, only {internal}-source fields contain the degree-zero (monopole) term. Accordingly, $B_+^{\ast}$ is invertible only on the zero-mean subspace $L^2(\SS)/\langle 1\rangle$.
\end{remark}

For later reference, we denote by $\SS_r=\{x\in\R^3:|x|=r\}$ and $\BB_r=\{x\in\R^3:|x|<r\}$ the sphere and ball of radius $r>0$, respectively. The spherical specifications discussed above, in particular the spherical harmonic expressions, are well-known. We mainly recapitulated them to highlight the relation to the more general $B$-operators, which are not as common in this context, and to set the stage for the upcoming studies that focus on the sphere and ball.

\section{Nonuniqueness and uniqueness for patches}\label{sec:localized}

Let $U$ be a fixed open subregion of $\SS$ (a ``patch''). Throughout the paper, we assume that $U$ has connected Lipschitz boundary and we denote its open complement by
\[
	V=\operatorname{int}_{\SS}(\SS\setminus U)=\SS\setminus \overline{U},
\]
assuming that it is non-empty. Whenever functions on $V$ are used as functions on the full sphere, they are understood as zero extensions to $\SS$. To make this clearer and more consistent, we define the function spaces $L^2(V)=\{f\in L^2(\SS):f=0\text{ a.e. on }U\}$ and
\[
	\widetilde H^1(V)=\{h\in H^1(\SS): h=0\ \text{a.e. on }U\}.
\]
Additionally, we define the auxiliary space $\widetilde H^1_{\text{const}}(V)=\{h\in H^1(\SS): h=\text{constant}\ \text{a.e. on }U\}$. We study internal-external field separation from data that are only available on $U$. That is, we are confronted with the question \eqref{eqn:probform} from the introduction.

\subsection{Localized separation is non-unique in \texorpdfstring{general}{general}}\label{subsec:localized-impossible}

The following theorem shows that, without additional assumptions, the answer to uniqueness is negative.

\begin{theorem}[Non-uniqueness for patch data]\label{thm:nonunique}
The restriction operator
\begin{align}
\mathcal A_U:\; \mathcal H_+(\SS)\times \mathcal H_-(\SS)\longrightarrow L^2(U)^3,\qquad
\mathcal A_U(\f_+,\f_-)=(\f_+ + \f_-)|_U,\label{eqn:aop}
\end{align}
has a nontrivial null space, i.e., $\ker(\mathcal A_U) \not=\{(0,0)\}$.
\end{theorem}

The explicit expression of the nullspace is provided in Corollary \ref{cor:nullspace} in Appendix \ref{app:null}.

\begin{proof}
The proof is based on layer-potential representations of Hardy components and on the characterization of $\mathcal H_\pm(\SS)$ via the $B$-operators from Section~\ref{subsec:spherical-hhd}. Let 
\begin{align}
\f_+ &= B_+^{\ast} f=\eta (K^{\ast}-\tfrac{1}{2})f + \grads S f,\\
\f_- &= B_-^{\ast}g=\eta(K^{\ast}+\tfrac{1}{2})g+\grads S g.
\end{align}
Then, $(\f_+ + \f_-)|_U = 0$ if and only if
\begin{numcases}{}
(K^{\ast}-\tfrac{1}{2})f + (K^{\ast}+\tfrac{1}{2})g = (K^{\ast}+\tfrac{1}{2}) (f+g) -f =0  \quad\text{on}\quad U,\label{eqn:loccond1}\\
\grads S f + \grads S g= \grads S (f+g) = 0 \quad \text{on}\quad U,
\label{eqn:loccond2}
\end{numcases}
i.e., the tangential part and normal part of $\f_+ + \f_-$ vanish on $U$. Condition \eqref{eqn:loccond2} implies that $S (f+g)$ is constant on $U$. Thus,
\begin{align}
	f+g= S^{-1}h \label{eqn:vanish_condition_2}
\end{align}
for an arbitrary $h\in{H}^1(\SS)$ that is constant on $U$, since $S$ is bounded and invertible from $L^2(\SS)$ to $H^1(\SS)$. From \eqref{eqn:loccond1} we then obtain
\begin{align}
f|_{U}= [(K^{\ast}+\tfrac{1}{2})S^{-1}h]|_{U}.  \label{eqn:vanish_condition_1}
\end{align}
On $V$, $f$ can be chosen to be any square-integrable function. This yields the nontriviality of the nullspace. Combining the observations from above, we also directly obtain the explicit representation of the nullspace as stated in Corollary \ref{cor:nullspace} in Appendix \ref{app:null}.
\end{proof}

\begin{remark}[Geophysical interpretation]\label{rem:geo-nonunique}
In particular, Theorem~\ref{thm:nonunique} states that there exist nonzero pairs $(\f_+,\f_-)\in \mathcal H_+(\SS)\times \mathcal H_-(\SS)$ such that $(\f_+ + \f_-)|_U = 0$ and, hence, internal-external field separation from information restricted to $U$ is \emph{not} unique in general.
	
This illustrates a basic limitation of localized geomagnetic observations: when data are only available on a restricted subregion of the sphere (e.g.\ a continental-scale array footprint), there may exist internal and external contributions that annihilate each other on the observation domain and thus do not contribute to the measured data. 
\end{remark}

\subsection{An altitude/analyticity condition restores uniqueness
}\label{subsec:localized-conditioned}

The non-uniqueness in Theorem~\ref{thm:nonunique} shows that patch data alone are insufficient for unconditional internal-external separation within the full classes $\mathcal H_+(\SS)$ and $\mathcal H_-(\SS)$. In many geophysical settings, however, the external current systems are not arbitrarily close to the Earth's surface. Near-surface air is typically treated as an effective insulator at the frequencies of interest, and the dominant external sources of the geo-electromagnetic system are associated with ionospheric and magnetospheric current systems, which are located at altitudes typically $\gtrsim 60\,\mathrm{km}$ above the surface \cite{Campbell2003}. With $a$ denoting the Earth's radius, an altitude $h$ above the Earth's surface corresponds to $r=1+h/a$ in our unit-sphere normalization (i.e., for  $a\approx 6371\,\mathrm{km}$, $h\approx60\,\mathrm{km}$ gives $r\approx 1.01$). This motivates imposing an altitude constraint (or, equivalently, an analyticity constraint) on the external-source component, meaning that the annulus $A_{a,a+h}=\{x\in\R^3:a<|x|<a+h\}$ (or the annulus $A_r=\{x\in\R^3: 1<|x|<r\}$ in our normalized setting) is considered source-free. Internal sources may still be arbitrarily close to the measurement surface. Thus, our results also hold for measurements at the Earth's surface. We show that such a constraint restores uniqueness of the localized internal-external field separation. However, it does not remove the fundamental instability of the inverse problem.

\begin{definition}[External-source altitude class]\label{def:Hplus_r}
For a fixed \(r>1\), we define \(\mathcal H_+^{(r)}(\SS)\subset \mathcal H_+(\SS)\) as the class of external-source fields on \(\SS\) whose associated scalar potential extends harmonically to the ball \(\BB_r=\{x\in\mathbb R^3:|x|<r\}\). 
More precisely,
\[
    \mathcal H_+^{(r)}(\SS)
    =
    \left\{
    \nabla u|_{\SS}:\Delta u=0\ \text{in }\BB_r,\ \|(\nabla u)_+^M\|_{L^2(\SS_r)^3}<\infty,\ u|_{\SS_r}\in H^1(\SS_r)
    \right\}.
\]
\end{definition}

Any \(\f_+\in\mathcal H_+^{(r)}(\SS)\) has an associated scalar potential \(u_+\) that is harmonic in \(\BB_r\). 
Let \(S_r\) denote the boundary single-layer operator on \(\SS_r\), with the same normalization as \(S\). That is, on $\SS_r$, we have $S_r\varphi=\frac{1}{4\pi}\int_{\SS_r}\frac{\varphi(y)}{|\cdot-y|}\,d\sigma(y)$ for $\varphi\in L^2(\SS_r)$. The spherical harmonic representation of \(S_r\), as indicated in \eqref{eq:SK-eigs} for $S$, shows that \(S_r:L^2(\SS_r)\longrightarrow H^1(\SS_r) \) is an isomorphism. Hence there exists a unique density \(\varphi_+=S_r^{-1}(u_+|_{\SS_r})\in L^2(\SS_r) \) such that \[ u_+(x)=\frac{1}{4\pi}\int_{\SS_r}\frac{\varphi_+(y)}{|x-y|}\,d\sigma(y), \qquad x\in \BB_r . \] Moreover, \[ \|u_+|_{\SS_r}\|_{H^1(\SS_r)} \le C(r)\|\varphi_+\|_{L^2(\SS_r)}, \] and, since \(S_r\) is an isomorphism, the converse bound also holds. Thus \(\|\varphi_+\|_{L^2(\SS_r)}\) is equivalent to the \(H^1(\SS_r)\)-norm of the 
trace \(u_+|_{\SS_r}\). Conversely, every potential $u_+$ defined via such a single-layer construction is harmonic in \(\BB_r\) and naturally leads to a member of \(\mathcal H_+^{(r)}(\SS)\).

\begin{theorem}[Uniqueness for patch data under an altitude/analyticity constraint]\label{thm:unique_r}
Let $r>1$ be fixed and denote by $\mathcal A_U^{(r)}:\mathcal H_+^{(r)}(\SS)\times \mathcal H_-(\SS)\to L^2(U)^3$ the restriction of $\mathcal A_U$ to $\mathcal H_+^{(r)}(\SS)\times \mathcal H_-(\SS)$. Then, $\mathcal A_U^{(r)}$ has only the trivial nullspace
\begin{align}
	\ker(\mathcal A_U^{(r)})=\big\{(0,0)\big\}.\nonumber
\end{align} 
\end{theorem}

\begin{proof}
Take $(\f_+,\f_-)\in \mathcal H^{(r)}_+(\SS)\times \mathcal H_-(\SS)$ such that $(\f_+ + \f_-)|_U=0$. By the definition of the underlying spaces, there exist functions $u_+$ and $u_-$ that are harmonic in $\BB_r$ and $\R^3\setminus\overline{\BB}$, respectively, such that $\f_+=\nabla u_+|_\SS$ and $\f_-=\nabla u_-|_\SS$. Therefore, the function $w=u_++u_-$ is harmonic in the annulus $A_r=\{x\in \R^3:1<|x|<r\}$. 
Due to the assumption $(\f_+ + \f_-)|_U=0$, its trace additionally satisfies $\nabla w|_U=0$. 
Therefore, $\nabla w$ can be expressed as the Abel-Poisson transform over the spheres $\SS$ and $\SS_r$ of a function that vanishes on an open patch $U\subset\SS$, and $w$ can be continued harmonically in an open environment of $U$. As a consequence, the trace of $w$ on $U$ can be understood in the classical sense and is real-analytic. We get $\nabla_\SS w=0$ and $\partial_\eta^+ w=0$ on $U$. The former implies that $w$ is constant on $U$ and leads to the following Cauchy data on the patch $U$:
\begin{align}
w=\text{const. on }U,\qquad \partial_\eta^+ w=0\ \text{ on }U .
\end{align}
From a Holmgren/analyticity-type uniqueness argument as in \cite[Lemma~1]{atfeh10}, it follows that $w$ is constant in $A_r$ and, hence, $\nabla w= 0$ in the annulus $A_r$. Taking the trace of $\nabla w$ on $\SS$ yields $\f_++\f_-=\nabla w|_\SS=0$. Thus, $\f_+=-\f_-$ on $\SS$. Since the spherical Hardy spaces $\mathcal H_+(\SS)$ and $\mathcal H_-(\SS)$ are orthogonal to each other, their intersection is trivial. Hence, $\f_+=\f_-=0$.
\end{proof}

\begin{remark}
	In particular, Theorem~\ref{thm:unique_r} states that there exist \emph{no} nonzero pairs $(\f_+,\f_-)\in \mathcal H_+^{(r)}(\SS)\times \mathcal H_-(\SS)$ such that $(\f_+ + \f_-)|_U = 0$ and, hence, internal-external field separation from information restricted to $U$ is unique in these spaces.
\end{remark}

In planetary-scale geomagnetic modelling, imposing a spherical bandlimit is standard. This provides a particularly transparent special case of the altitude/analyticity condition. In such a setting, Theorem \ref{thm:unique_r} applies directly and yields the following corollary.

\begin{corollary}[Bandlimited external-source fields]\label{cor:bandlimited}
If the external-source component $\f_+\in\mathcal H_+(\SS)$ is bandlimited with respect to its spherical harmonic expansion, i.e.,
\begin{align*}
 \f_+  \in \mathrm{span}\{\mathbf G^{\mathrm{ext}}_{n,k}: n\leq N,\,|k|\leq n\} 
\end{align*}
for some fixed integer $N$, then it satisfies the condition of Theorem~\ref{thm:unique_r}. Hence,  the restriction of $\mathcal A_{U}$ to $\mathrm{span}\{\mathbf G^{\mathrm{ext}}_{n,k}: n\leq N,\,|k|\leq n\}\times \mathcal H_{-}(\SS)$ is injective.
\end{corollary}

\begin{remark}[Geophysical interpretation]\label{rem:geo-nonunique2}
The above illustrates that geophysically reasonable assumptions can suffice to make internal--external field separation from regional data theoretically possible. Moreover, bandlimited signals expanded in terms of vector spherical harmonics, as commonly assumed in global geomagnetic modelling, also yield uniqueness (as indicated in Corollary \ref{cor:bandlimited}). However, bandlimitedness is usually an idealized assumption that is not satisfied by real geophysical fields, such that this setting should be viewed as a first approximate simplification rather than as a resolution of the underlying nonuniqueness. 

Slepian functions as, e.g., in \cite{plamazger24,plattnersimons14b,Plattner2017} provide a related bandlimited approach that additionally accounts for the spatial localization of the data and provides an intrinsic regularization. 
Spherical elementary current systems (SECS; see, e.g., \cite{Vanhamki2019,Pulkkinen2014}) provide another localized parametrization strategy, in which internal and external contributions are represented implicitly by finitely many elementary equivalent-current systems placed at prescribed locations or altitudes. Although the equivalent-current modelling underlying SECS differs from the potential-field framework considered here, the prescribed placement of the elementary systems can encode an
altitude-type source prior of the type used in Theorem \ref{thm:unique_r}. From the viewpoint of the present analysis, SECS-based approaches can therefore be viewed as finite-dimensional regularization strategies for the regional separation problem, while the underlying ill-posedness of the unrestricted problem remains a relevant consideration.

\end{remark}

\section{Instability and conditional stability for patches}\label{sec:instab}

Theorem~\ref{thm:unique_r} shows that an altitude/analyticity constraint can restore uniqueness. However, uniqueness does not imply stability in the infinite-dimensional function space setup. Even under the altitude/analyticity condition, the inverse problem remains ill-posed on a patch. The mechanism behind this instability is the following local density property: $\mathcal H_+(\SS)|_{U} \cap \mathcal H_-(\SS)|_{U}$ is dense in $\mathcal H_+(\SS)|_{U}$ with respect to the $L^2(U)^3$-norm. In particular, the same approximation property applies to $\mathcal H_+^{(r)}(\SS)|_{U}$ since $\mathcal H_+^{(r)}(\SS)\subset \mathcal H_+(\SS)$. Therefore, the inverse map of $\mathcal A_U^{(r)}$ cannot be bounded, even under the altitude/analyticity condition. Nevertheless, if one imposes additional source conditions on $\f_\pm$, one can still derive a quantitative estimate of logarithmic type.
\subsection{Instability despite uniqueness}\label{subsubsec:instability}

The density of $\mathcal H_+(\SS)|_{U} \cap \mathcal H_-(\SS)|_{U}$ in $\mathcal H_+(\SS)|_{U}$ follows from the following analysis of the nullspace of $\mathcal A_{U}$.

\begin{theorem}\label{thm:projnullspace}
  The space $\mathrm{proj}_{\mathcal H_+(\SS)}(\ker(\mathcal A_{U}))$ is dense in $\mathcal H_+(\SS)$. That is, for any $\f_+\in\mathcal H_+(\SS)$ and $\eps>0$, there exist $\tilde{\f}_\pm\in \mathcal H_\pm(\SS)$ such that
  \begin{align}\label{eqn:desity_1}
      (\tilde{\f}_++\tilde{\f}_-)|_U=0 && \text{and} &&\|\tilde{\f}_+-\f_+\|_{L^2(\SS)^3}\leq \eps.
  \end{align}
\end{theorem}

\begin{proof}
Equation \eqref{eqn:desity_1} is a direct consequence of the mentioned density. Thus, we only need to prove that $\mathrm{proj}_{\mathcal H_+(\SS)}(\ker(\mathcal A_{U}))$ is dense in $\mathcal H_+(\SS)$. Since the operator $B_+^{\ast}$ is bounded from $L^2(\SS)$ to $L^2(\SS)^3$, it suffices to show that the scalar functions $f$ satisfying conditions \eqref{eqn:vanish_condition_2} and \eqref{eqn:vanish_condition_1} form a dense subset of $L^2(\SS)$ (also compare the kernel representation from Corollary \ref{cor:nullspace}). For this, no inversion of the operator $B_+^*$ is required. Thus, the spherical harmonic degree zero causes no obstruction, since $B_+^\ast 1=0$ and, hence, it is simply contained in the kernel of this parametrization.

Again, note that the two equations \eqref{eqn:vanish_condition_2} and \eqref{eqn:vanish_condition_1} impose no particular constraints on $f|_{V}$. According to Corollary \ref{cor:nullspace}, this means that we can choose $f$ freely on $V$. In this nullspace representation, $h$ may be chosen in $\widetilde H^1_{\mathrm{const}}(V)$. For the density argument below, it is enough to use the subspace $\widetilde H^1(V)\subset\widetilde H^1_{\mathrm{const}}(V)$. Via the choice $g=S^{-1}h-f$, equation \eqref{eqn:vanish_condition_2} is then always satisfied. Therefore, it remains to show that the restrictions $f|_{U}$, for $f$ satisfying~\eqref{eqn:vanish_condition_1}, are dense in $L^2(U)$. We denote this set by 
\begin{align}
    \tilde{M}=\left\{\, \big[(K^{\ast}+\tfrac{1}{2})S^{-1}h\big]\big|_{U} : h\in \widetilde H^1(V)\, \right\}\subset L^2(U).
\end{align}
We prove its density by showing that $\tilde{M}^{\perp}=\{0\}$. 

Let $f^{\ast}\in \tilde{M}^{\perp} \subset L^2(U)$, and denote by $\bar{f^{\ast}}\in L^2(\SS)$ the zero extension of $f^{\ast}$ to the full sphere $\SS$. We use the operator
\[
	A=(K^\ast+\tfrac12)S^{-1}:H^1(\SS)\longrightarrow L^2(\SS).
\]
On the sphere, $S$ and $K^\ast=K$ diagonalize in the spherical harmonic basis, more precisely,
\[
	A Y_{n,k}=-(n+1)Y_{n,k},\qquad n\ge0.
\]
Thus, $A$ is a first-order spherical harmonic multiplier with domain $H^1(\SS)$ when regarded as an unbounded operator in $L^2(\SS)$, and it extends continuously from $L^2(\SS)$ to $H^{-1}(\SS)$ by duality. Hence, for every admissible $h\in \widetilde H^1(V)$,
\begin{align}
    0&=\left\langle Ah, \bar{f^{\ast}}\right\rangle_{L^2(\SS)}
    =\left\langle h, A\bar{f^{\ast}} \right\rangle_{H^1(\SS),H^{-1}(\SS)}. \label{eqn:dual_main}
\end{align}
In particular, \eqref{eqn:dual_main} holds for all $h\in \widetilde H^1(V)$. Therefore
\begin{align}
	A\bar{f^{\ast}}=0 \quad \text{in } H^{-1}(V),\label{eqn:Afzerosobolev}
\end{align}
where $H^{-1}(V)$ is understood as $(\widetilde H^1(V))'$. while $\bar{f^{\ast}}=0$ in $L^2(V)$ by construction. It remains to show that these two Cauchy data force $\bar{f^{\ast}}=0$ on $\SS$.

Since $\bar f^{\ast}\in L^2(\SS)$, the density $\mu=S^{-1}\bar f^{\ast}$ is only in $H^{-1}(\SS)$.  We therefore regularize by a local zonal approximate identity. Let the operator $T_\varepsilon:L^2(\SS)\to L^2(\SS)$, or $T_\varepsilon:H^{-1}(\SS)\to H^{-1}(\SS)$ where applicable (in that case, the integral should be understood as the corresponding dual product), be given by
\begin{align}
	T_\varepsilon g(x)=\int_\SS k_\varepsilon(x\cdot y) g(y)d\sigma(y)\in C^\infty(\SS),\label{eqn:tepsrep}
\end{align} 
where $k_\varepsilon\in C^\infty([-1,1])$ is a smooth zonal kernel with $\textnormal{supp}(k_\varepsilon)\subset(\cos \varepsilon,1]$. 
We choose $k_\varepsilon$ such that $T_\varepsilon g\to g$ as $\varepsilon\to 0$. The latter convergence is meant in $L^2(\SS)$ for every $g\in L^2(\SS)$ and in $H^{-1}(\SS)$ for every $g\in H^{-1}(\SS)$. Since $T_\varepsilon$ is a zonal convolution operator, it is diagonal in the spherical harmonic basis and, therefore, commutes with the operators $S$, $S^{-1}$, and $K^\ast$ (which are spherical spectral multipliers) on their natural  domains.

Set $\bar f^{\ast}_\varepsilon=T_\varepsilon\bar f^{\ast}$ and
$\mu_\varepsilon=S^{-1}\bar f^{\ast}_\varepsilon$. Then, we get $\bar f^{\ast}_\varepsilon,\mu_\varepsilon\in C^\infty(\SS)$ and
\begin{align}
(K^{\ast}+\tfrac{1}{2})\mu_\varepsilon
=(K^{\ast}+\tfrac{1}{2})S^{-1}\bar f^{\ast}_\varepsilon
=T_\varepsilon\big((K^{\ast}+\tfrac{1}{2})S^{-1}\bar f^{\ast}\big).
\end{align}
Moreover, since $\bar f^{\ast}=0$ and $A\bar f^\ast=0$ on $V$ (remember property \eqref{eqn:Afzerosobolev}), locality of $T_\varepsilon$ gives
\[
	\bar f^{\ast}_\varepsilon=0,\qquad (K^{\ast}+\tfrac12)\mu_\varepsilon=0
	\quad\text{on}\quad
	V_\varepsilon=\{x\in V:\operatorname{dist}_{\SS}(x,\partial U)>\varepsilon\}.
\]
For sufficiently small $\varepsilon>0$, $V_\varepsilon$ contains a non-empty open patch. Next, set $u_\varepsilon=\mathcal{S}\mu_\varepsilon$ to be the single-layer potential of $\mu_\varepsilon$, in the sense
\begin{align}
	u_\varepsilon(x)=\mathcal{S}\mu_\varepsilon(x)=\frac{1}{4\pi}\int_\SS\frac{1}{|x-y|}\mu_\varepsilon(y)\,d\sigma(y),\qquad x\in\R^3\setminus\overline{\BB}.
\end{align} 
The classical jump relations for smooth single-layer densities imply
\[
	u_\varepsilon|_{\SS}=\bar f^{\ast}_\varepsilon,\qquad
	\partial_\eta^- u_\varepsilon=(K^{\ast}+\tfrac12)\mu_\varepsilon,
	\quad\text{on }\SS
\]
with the sign convention of \cite{BGK} that is used throughout this manuscript. Hence $u_\varepsilon=0$ and $\partial_\eta^- u_\varepsilon=0$ on $V_\varepsilon$. Since $u_\varepsilon$ is harmonic in $\R^3\setminus\overline{\BB}$ and vanishes at infinity, Cauchy uniqueness for harmonic functions across the analytic boundary patch $V_\varepsilon$ yields $u_\varepsilon=0$ in $\R^3\setminus\overline{\BB}$ (for instance via the Holmgren/analyticity argument in \cite[Lemma~1]{atfeh10}). Taking the exterior trace gives $\bar f^{\ast}_\varepsilon=0$ on $\SS$. Finally, $\bar f^{\ast}_\varepsilon\to\bar f^{\ast}$ in $L^2(\SS)$ as $\varepsilon\to0$, so that $\bar f^{\ast}=0$. Thus, $\tilde{M}^\perp=\{0\}$, and the desired density follows.

\end{proof}

\begin{corollary}\label{thm:instability}
The set $\mathcal H_+(\SS)|_{U} \cap \mathcal H_-(\SS)|_{U}$ is a dense subset of $\mathcal H_+(\SS)|_{U}$. That is, for any $\f_+\in\mathcal H_+(\SS)$ and $\eps>0$, there exists $\f_-\in \mathcal H_-(\SS)$, such that
  \begin{align}\label{eqn:desity_2}
      \|\f_+ - \f_-\|_{L^2(U)^3}\leq \eps.
  \end{align}
\end{corollary}
\begin{proof}
    We can simply choose $\f_-=-\tilde{\f}_-$ where $\tilde{\f}_+$, $\tilde{\f}_-$ indicate the functions stated in~\eqref{eqn:desity_1}. Then it holds 
    \begin{align*}
         \|\f_+ - \f_-\|_{L^2(U)^3}&=  \|\f_+ + \tilde{\f}_-\|_{L^2(U)^3} \leq \|\f_+ - \tilde{\f}_+\|_{L^2(U)^3} + \|\tilde{\f}_+ + \tilde{\f}_-\|_{L^2(U)^3}
         \\&=\|\tilde{\f}_+ - \f_+\|_{L^2(U)^3}\leq \|\tilde{\f}_+-\f_+\|_{L^2(\SS)^3}\leq \eps,
    \end{align*}
    which is the desired statement.
\end{proof}

\begin{remark}
    The commutation requirement for the mollifier and the single- and double-layer operators at the end of the proof of Theorem \ref{thm:projnullspace} is a main obstacle to extending the previous results to general $C^2$-smooth surfaces or to Lipschitz surfaces.
\end{remark}

Since $\mathcal H_+^{(r)}(\SS)\subset \mathcal H_+(\SS)$, equation~\eqref{eqn:desity_2} holds in particular for any $\f_+\in \mathcal H_+^{(r)}(\SS)$. This means that, even under the altitude/analyticity condition, internal-external field separation cannot be stable.

\begin{corollary}[Instability on patches]\label{thm:instabilityforHr}
There exists \emph{no} continuous function $\Phi:\R_+\to\R_+$ with $\lim_{\varepsilon\to 0}\Phi(\varepsilon)=0$ such that
\begin{align}
\|\f_+\|_{L^2(\SS)^3}+\|\f_-\|_{L^2(\SS)^3}
\le \Phi\Big(\|(\f_+ + \f_-)|_U\|_{L^2(U)^3}\Big),
\quad\text{for all }\f_+\in \mathcal H_+^{(r)}(\SS),\ \f_-\in \mathcal H_-(\SS).
\end{align}
\end{corollary}

\begin{remark}[Implications for practice]\label{rem:practice}
The previous results show that localized internal-external potential field separation remains challenging in practice. An altitude/analyticity assumption can remove the ambiguity of non-uniqueness, but small perturbations of patch data may still cause large changes in the separated components. A combination of additional prior information and/or regularization methods is, therefore, typically necessary. 
\end{remark}

\subsection{A quantitative conditional estimate}

Instability of harmonic continuation as well as general elliptic Cauchy problems has long been known and quantitative estimates have been supplied, e.g., in \cite{aleron09,bou10,isakov12,ruesal19}. Since harmonic continuation is also the underlying driver of the internal-external field separation problem, we can make use of these results in our setup. In the following, we provide a conditional logarithmic stability estimate for the source-free shell $A_r$. To formulate the required additional source condition, we use the single-layer representation as indicated after Definition~\ref{def:Hplus_r}: for $\f_+=\nabla u_+|_\SS\in \mathcal H_+^{(r)}(\SS)$, we write
\[
	u_+(x)=\frac{1}{4\pi}\int_{\SS_r}\frac{\varphi_+(y)}{|x-y|}\,d\sigma(y),
	\qquad x\in\BB_r,
\]
with $\varphi_+=S_r^{-1}(u_+|_{\SS_r})\in L^2(\SS_r)$. The condition $\|\varphi_+\|_{L^2(\SS_r)}\le M$ in Theorem~\ref{thm:condstab} states an a priori bound on the harmonic continuation of $u_+$ in the annulus $A_r$.
In addition, the estimate below uses a Sobolev source prior
$\|\f_-\|_{H^\tau(\SS)^3}\le M$, with $\tau\geq\frac12$, on the internal-source
contribution. This condition guarantees the $H^2(A_{1+s})$ a priori regularity required
by the elliptic Cauchy stability estimate from Lemma \ref{lem:carleman} and excludes arbitrarily oscillatory internal fields.

\begin{theorem}[Conditional logarithmic stability on patches with a Sobolev source prior] \label{thm:condstab}
{
	Fix an arbitrary $\alpha\in(0,1)$, an arbitrary $\tau\geq\frac12$, and $M>0$. Then, for every $r>1$, there exist constants $C_1=C_1(r,U,\alpha,\tau)>0$, $C_2=C_2(r,U,\alpha,\tau)>0$, and a threshold $\delta=\delta(r,U,\alpha,\tau,M)>0$ such that the following holds true:
    for all $\f_+\in \mathcal H_+^{(r)}(\SS)$, $\f_-\in \mathcal H_-(\SS)$ that satisfy the source conditions
	\[
		\|\f_-\|_{H^\tau(\SS)^3}\le M,
		\qquad
		\|\varphi_+\|_{L^2(\SS_r)}\le M,
		\qquad
		\|(\f_+ + \f_-)|_U\|_{L^2(U)^3}<\delta,
	\]
	we have
	\begin{align}
		\|\f_+\|_{L^2(\SS)^3}+\|\f_-\|_{L^2(\SS)^3}
		\le \Phi\left(\|(\f_+ + \f_-)|_U\|_{L^2(U)^3}\right),\label{eqn:logstab}
	\end{align}
	with
	\begin{align}
		\Phi(t)=
		\begin{cases}
		C_1 M\left|\ln\left(\dfrac{C_2 M}{t}\right)\right|^{-\alpha\vartheta_\tau},
		& 0<t<\delta,\\[0.3em]
		0, & t=0,
		\end{cases}
		\label{eqn:logstab2}
	\end{align}
	where $\vartheta_\tau=\frac{\tau}{\tau+\frac12}\in(0,1)$.
}
\end{theorem}

In contrast to Corollary~\ref{thm:instabilityforHr}, Theorem~\ref{thm:condstab} requires two source conditions: the altitude or analyticity bound $\|\varphi_+\|_{L^2(\SS_r)}\leq M$ for the external-source contribution and the Sobolev bound $\|\f_-\|_{H^\tau(\SS)^3}\leq M$ with $\tau\geq\frac12$ for the internal-source contribution. The latter condition has two roles. First, it provides the $H^2$ a priori bound on the harmonic continuation in the shell that is required in the logarithmic Cauchy estimate of \cite[Cor.~2.1]{bou10}. Second, it allows us to upgrade the resulting weak boundary stability estimate to an $L^2(\SS)^3$-estimate by Sobolev interpolation. Accordingly, the logarithmic exponent is reduced from $\alpha$ to $\alpha\vartheta_\tau$. 

Under this assumption, Theorem~\ref{thm:condstab} provides a quantitative estimate for the separation error. The key point is that the control is only of logarithmic type, through the function $\Phi(t)\sim |\ln t|^{-\alpha\vartheta_\tau}$. Thus, improving the data fit on the observation patch leads to only a very slow improvement in the recovered internal and external fields, rendering the localized separation problem practically unstable even under the altitude/analyticity condition.

\begin{proof}
	Let $\f_+=\nabla u_+|_\SS$ and $\f_- =\nabla u_-|_\SS$. As in the proof of Theorem \ref{thm:unique_r}, the function $w=u_++u_-$ is harmonic in the annulus $A_r$. For convenience, we set
	\begin{align}
		\varepsilon=\|\nabla w|_U\|_{L^2(U)^3}=\|(\f_++\f_-)|_U\|_{L^2(U)^3}.
	\end{align}
	We normalize $w$ by subtracting its patch average, i.e., we set
	\[
		\widetilde w=w-c_U ,\qquad c_U=\frac{1}{|U|}\int_U w(y)\,d\sigma(y).
	\]
	Then, \(\nabla\widetilde w=\nabla w\) in \(A_r\), \(\partial_\eta^+\widetilde w=\partial_\eta^+w\) on \(U\), and \(\int_U\widetilde w\,d\sigma=0\). Splitting the boundary gradient into tangential and normal components yields
	\[
		\|\nabla_\SS \widetilde w|_U\|_{L^2(U)^3}\le \varepsilon,
		\qquad
		\|\partial_\eta^+ \widetilde w|_U\|_{L^2(U)}\le \varepsilon.
	\]
		By the Poincar\'e inequality on the patch \(U\), this implies
		\begin{align}
			\|\widetilde w\|_{H^1(U)}\le C_U\,\varepsilon,
			\qquad
			\|\partial_\eta^+ \widetilde w\|_{L^2(U)}\le \varepsilon, \qquad \|\widetilde w\|_{H^1(U)}+\|\partial_\eta^+\widetilde w\|_{L^2(U)}
		\le (C_U+1)\varepsilon ,\label{eqn:smallnessest}
		\end{align}
     where the constant $C_U>0$ may depend on the patch $U$.
    
    Next, we derive estimates on the Sobolev norms of $u_+$ and $u_-$ in the annulus $A_{1+s}\subset A_r$, for $s=\tfrac{r-1}{4}$. For any $x\in A_{1+s}$ and $y\in \SS_r$ it holds $|x-y|\geq r-(1+s)=\tfrac{3}{4}(r-1)=\tilde{s}$. Observing that the surface area of the sphere $\SS_r$ is $|\SS_r|=4\pi r^2$, Cauchy's inequality and the assumption on $\varphi_+$ yield that, for $x\in A_{1+s}$,
	\begin{align}
		|u_+(x)|&\leq \int_{\SS_r}\frac{|\varphi_+(y)|}{|x-y|}d\sigma(y)\leq \frac{\sqrt{4\pi} r}{\tilde s}\|\varphi_+\|_{L^2(\SS_r)}\leq \frac{3\sqrt{4\pi} r}{\tilde{s}} M,
		\\|\nabla u_+(x)|&\leq \int_{\SS_r}\frac{|\varphi_+(y)|}{|x-y|^2}d\sigma(y)\leq \frac{\sqrt{4\pi} r}{\tilde{s}^2} M.
	\end{align}
	Integrating over the annulus $A_{1+s}$ yields
    \begin{align}
			\|u_+\|_{H^1(A_{1+s})}^2&\leq \frac{4\pi}{3\tilde{s}^4}\left((1+s)^3-1\right)(4\pi+36\pi\tilde{s}^2)r^2M^2\leq C_+(r)^2M^2\frac{1}{r-1},\label{eqn:uplusest}
	\end{align}
    for a constant $C_+(r)\geq c_0$, with $c_0>0$ being fixed and neither depending on $U$ nor $r$.
    
    Since $H^\tau(\SS)^3$ is continuously embedded in $ L^2(\SS)^3$, the assumption $\|\f_-\|_{H^\tau(\SS)^3}\le M$ implies $\|\f_-\|_{L^2(\SS)^3}\le \tilde{C}_- M$. Hence the spherical harmonic computation in Appendix~\ref{app:aux}  gives
    \begin{align}
			\|u_-\|_{H^1(A_{1+s})}^2&\leq  C_-(r)^2M^2,\label{eqn:uminusest}
	\end{align}
    for a constant $C_-(r)>0$. 
    
Hence, for the original function \(w=u_++u_-\), the estimates \eqref{eqn:uplusest} and \eqref{eqn:uminusest} provide
	\begin{align}
		\|w\|_{H^1(A_{1+s})}
		\le \|u_+\|_{H^1(A_{1+s})}+\|u_-\|_{H^1(A_{1+s})}
		\le (C_+(r)+C_-(r))\frac{M}{\sqrt{r-1}}.
	\end{align}
	The trace theorem provides us with a constant $C_{\rm tr}(r,U)\geq c_0$ such that
	\[
		|c_U|\le |U|^{-1/2}\|w\|_{L^2(U)}
		\le C_{\rm tr}(r,U)\|w\|_{H^1(A_{1+s})}.
	\]
	Combining the previous two estimates leads to 
	\begin{align}
		\|\widetilde w\|_{H^1(A_{1+s})}=\|w-c_U\|_{H^1(A_{1+s})}
		\le C_0(r,U)\frac{M}{\sqrt{r-1}},\label{eqn:tildewest}
	\end{align}
	where \(C_0(r,U)>0\) absorbs the constants $C_{\rm tr}(r,U)$, \(C_+(r)\), and \(C_-(r)\).

 We next record the $H^2$ a priori estimate that is needed in order to apply the logarithmic Cauchy stability estimate of \cite[Cor.~2.1]{bou10}. For the external-source potential $u_+$, differentiating the single-layer representation up to second order and using the positive distance between $A_{1+s}$ and $\SS_r$ gives
\[
\|u_+\|_{H^2(A_{1+s})}\le C_{+,2}(r)M.
\]
For the internal-source potential $u_-$, the Sobolev source prior gives the
corresponding shell estimate
\[
    \|u_-\|_{H^2(A_{1+s})}\le C_{-,2}(r,\tau)M,
    \qquad \tau\geq\frac12 .
\]
This follows from the same spherical-harmonic estimates as in Appendix~\ref{app:aux},
with two additional derivatives. The condition $\tau\geq1/2$ supplies precisely
the summability needed to control the $H^2$-norm in the exterior collar. Together with the estimate of the constant $c_U$ above, we obtain
\begin{align}
\|\widetilde w\|_{H^2(A_{1+s})}\le \tilde{C}_{0}(r,U,\tau)M.\label{eqn:h2prior-wtilde}
\end{align}

For the given $\alpha\in(0,1)$, let $C=C(r,U,\alpha)>0$ and $\delta_0=\delta_0(r,U,\alpha,\tilde{C}_{0}M)>0$ be the constants obtained from Lemma~\ref{lem:carleman}, applied with the a priori bound $\tilde{M}=\tilde{C}_{0}(r,U,\tau)M$. We choose $\delta=\delta(r,U,\alpha,\tau,M)>0$ sufficiently small so that $(C_U+1)\delta<\delta_0$ and $\frac{c\,\tilde{C}_{0}M}{(C_U+1)\varepsilon}>1$
for all $0<\varepsilon<\delta$. Then Lemma~\ref{lem:carleman}, together with \eqref{eqn:smallnessest} and \eqref{eqn:h2prior-wtilde}, gives the weak boundary estimate
	\begin{align}
		\|\nabla w|_\SS\|_{H^{-1/2}(\SS)^3}
		=\|\nabla\widetilde w|_\SS\|_{H^{-1/2}(\SS)^3}
		\le C(r,U,\alpha)\,\tilde{C}_{0}(r,U,\tau)M
		\left| \ln\left(\frac{\tilde{C}_{0}(r,U,\tau)M}{(C_U+1)\varepsilon}
		\right)
		\right|^{-\alpha}.\label{eqn:nablaw-est}
	\end{align}

 It remains to upgrade this weak boundary estimate to an $L^2(\SS)^3$ estimate. For every fixed $\tau\geq\frac12$, the altitude/analyticity condition and the bound $\|\varphi_+\|_{L^2(\SS_r)}\le M$ imply, by the spherical harmonic representation of the single-layer potential on $\SS_r$, that
\[
\|\f_+\|_{H^\tau(\SS)^3}\le C_{+,\tau}(r)M.
\]
Together with the source condition $\|\f_-\|_{H^\tau(\SS)^3}\le M$, this yields, 
\begin{align}
\|\nabla w|_\SS\|_{H^\tau(\SS)^3}
\le (1+C_{+,\tau})M.\label{eqn:htau-prior}
\end{align}
By Sobolev interpolation between $H^{-1/2}(\SS)^3$ and $H^\tau(\SS)^3$,
\[
\|\nabla w|_\SS\|_{L^2(\SS)^3}
\le C_{\rm int}(\tau)\,
\|\nabla w|_\SS\|_{H^{-1/2}(\SS)^3}^{\vartheta_\tau}
\|\nabla w|_\SS\|_{H^\tau(\SS)^3}^{1-\vartheta_\tau},
\qquad
\vartheta_\tau=\frac{\tau}{\tau+\frac12}.
\]
Combining this interpolation estimate with \eqref{eqn:nablaw-est} and \eqref{eqn:htau-prior} gives (with the constants $C_{\mathrm{int}}$, $C$, $\tilde{C}_0$, $C_{+,\tau}$ being absorbed into the constant $\tilde{C}_1$)
\begin{align}
\|\f_++\f_-\|_{L^2(\SS)^3}
\le \tilde{C}_1(r,U,\alpha,\tau)M
\left|\ln\left(\frac{C_2(r,U,\alpha,\tau)M}{\varepsilon}\right)\right|^{-\alpha\vartheta_\tau}.\label{eqn:l2g-est}
\end{align}
	Finally, on the sphere we have $\f_+\in \mathcal H_+(\SS)$, $\f_-\in \mathcal H_-(\SS)$ and $\mathcal H_+(\SS)\perp \mathcal H_-(\SS)$, it follows that $\|\f_++\f_-\|_{L^2(\SS)^3}^2=\|\f_+\|_{L^2(\SS)^3}^2+\|\f_-\|_{L^2(\SS)^3}^2$ and, therefore,
	\[
		\|\f_+\|_{L^2(\SS)^3}+\|\f_-\|_{L^2(\SS)^3}
		\le \sqrt{2}\,\|\nabla w|_\SS\|_{L^2(\SS)^3}.
	\]
	Combining this with \eqref{eqn:l2g-est}, and choosing $C_1=\sqrt{2}\tilde{C}_1$, yields the desired estimate \eqref{eqn:logstab}. The case $\varepsilon=0$ follows from Theorem~\ref{thm:unique_r}.
    
\end{proof}

\begin{remark}[Geophysical interpretation of the conditional estimate] Theorem~\ref{thm:condstab} should be read as a conditional stability statement. The source-free shell \(A_r\) represents the assumption that the external current systems are separated from the observation surface by a positive altitude. The additional Sobolev bound with \(\tau\geq1/2\) is a mild smoothness prior on the internal-source contribution. This threshold is mainly technical: it provides the boundary regularity needed to pass from weak Cauchy stability to an \(L^2\)-estimate on the separated fields. From a modelling point of view, this is a mild restriction: it excludes arbitrarily oscillatory internal fields, but it is consistent with standard smoothness assumptions or bandlimited representations. 

Under these priors, a small data mismatch on \(U\) implies only a logarithmically small global separation error. Thus, the regional separation problem remains strongly ill-conditioned in practice, even in the uniqueness regime. The constants depend on the shell thickness, the patch geometry, and the chosen smoothness prior. Their precise behaviour as \(r\downarrow1\), i.e., as it approaches the nonuniqueness regime, is not tracked here. \end{remark}

\section{Discussion: implications for geomagnetic modelling}\label{sec:discussion}

The results of Section~\ref{sec:localized} sharpen a basic distinction between the global and regional data in the internal-external potential field separation problem. For global data on a closed surface, the boundedness of the Hardy-Hodge projections yields a stable separation into external-source, internal-source, and tangential divergence-free components. In contrast, if measurements are available only on a proper surface patch (as is the case, e.g., for continental-scale arrays), Theorem~\ref{thm:nonunique} shows that internal and external contributions need not be determined uniquely from the local data. Thus, the localized problem is not merely a truncated version of the global one but a problem with a genuinely different structural property.

A physically natural way to reduce this ambiguity is to incorporate prior information on the location of the external sources. This motivates the altitude/analyticity constraint encoded by the class $\mathcal H_+^{(r)}(\SS)$ in Definition~\ref{def:Hplus_r}. In the unit-sphere normalization, $r-1$ measures the radial gap between the observation surface and the nearest allowed external sources. Under such a constraint, Theorem~\ref{thm:unique_r} establishes uniqueness of the localized separation, and a bandlimited external-source field assumption provides a particularly restrictive special case (Corollary~\ref{cor:bandlimited}).

Uniqueness, however, does not imply stability. Corollaries~\ref{thm:instability} and~\ref{thm:instabilityforHr} show that even under an altitude/analyticity constraint, localized separation remains ill-posed: small perturbations of patch data may correspond to large changes in the separated components. Consequently, any practical separation workflow for patch data necessarily relies on additional priors or regularization. Common choices in geomagnetic practice include spectral truncation/bandlimiting, parametric external-field models, minimum-energy or smoothness penalties, spatio-spectral localization. 

We emphasize a point of interpretation: throughout this paper, “internal/external” refers to source location relative to the observation surface, not to a classification by physical processes. In particular, for time-varying geomagnetic fields, the internal component may largely represent the induced (secondary) field generated within the conducting Earth in response to primary external current systems. Conversely, the same current system may be classified differently depending on the observation altitude: for example, ionospheric currents are often treated as external with respect to ground observations but as internal with respect to low-Earth-orbit satellite measurements. Combining such complementary observation modalities is an appealing topic, but it lies beyond the scope of the present analysis.

\section*{Acknowledgements}
The authors are grateful to the anonymous referee for very valuable comments on an earlier version of this manuscript.

\small
\bibliography{example}

@InCollection{Vanhamki2019,
  author    = {Vanham\"{a}ki,  H. and Juusola,  L.},
  booktitle = {Ionospheric Multi-Spacecraft Analysis Tools},
  publisher = {Springer},
  title     = {Introduction to Spherical Elementary Current Systems},
  year      = {2020},
  editor    = {Dunlop, M. and L\"uhr, H. },
}

@article{Pulkkinen2014,
  title = {Separation of the geomagnetic variation field on the ground into external and internal parts using the spherical elementary current system method},
  volume = {55},
  ISSN = {1880-5981},
  url = {http://dx.doi.org/10.1186/BF03351739},
  DOI = {10.1186/bf03351739},
  number = {3},
  journal = {Earth,  Planets and Space},
  publisher = {Springer Science and Business Media LLC},
  author = {Pulkkinen,  A. and Amm,  O. and Viljanen,  A. and BEAR working group },
  year = {2003},
  month = June,
  pages = {117–129}
}

@article{Torta2019,
  title = {Modelling by Spherical Cap Harmonic Analysis: A Literature Review},
  volume = {41},
  ISSN = {1573-0956},
  url = {http://dx.doi.org/10.1007/s10712-019-09576-2},
  DOI = {10.1007/s10712-019-09576-2},
  number = {2},
  journal = {Surveys in Geophysics},
  publisher = {Springer Science and Business Media LLC},
  author = {Torta,  J. M.},
  year = {2019},
  month = nov,
  pages = {201–247}
}

@Book{backus1996,
  author    = {Backus, G. and Parker, R. and Constable, C.},
  title     = {Foundations of Geomagnetism},
  isbn      = {978-0521017336},
  note      = {Paperback},
  publisher = {Cambridge University Press},
  address   = {Cambridge},
  isbn-10   = {0521017335},
  year      = {1996},
}

@Article{Alken2021,
  author    = {Alken, P. and Thébault, E. and Beggan, C. and others},
  title     = {International Geomagnetic Reference Field: the thirteenth generation},
  doi       = {10.1186/s40623-020-01288-x},
  issn      = {1880-5981},
  url       = {http://dx.doi.org/10.1186/s40623-020-01288-x},
  volume    = {73},
  journal   = {Earth, Planets and Space},
  publisher = {Springer Science and Business Media LLC},
  year      = {2021},
}

@Book{Campbell2003,
  author    = {Campbell, W. H.},
  title     = {Introduction to Geomagnetic Fields},
  doi       = {10.1017/cbo9781139165136},
  isbn      = {9781139165136},
  publisher = {Cambridge University Press},
  url       = {http://dx.doi.org/10.1017/CBO9781139165136},
  month     = apr,
  year      = {2003},
}

@article{Olsen2012,
  title = {Satellite Geomagnetism},
  volume = {40},
  ISSN = {1545-4495},
  url = {http://dx.doi.org/10.1146/annurev-earth-042711-105540},
  DOI = {10.1146/annurev-earth-042711-105540},
  number = {1},
  journal = {Annual Review of Earth and Planetary Sciences},
  publisher = {Annual Reviews},
  author = {Olsen,  N. and Stolle,  C.},
  year = {2012},
  month = may,
  pages = {441–465}
}

@Article{BGK,
  author  = {Baratchart, L. and Gerhards, C. and Kegeles, A.},
  journal = {SIAM Journal on Mathematical Analysis},
  title   = {Decomposition of ${L}^2$-vector fields on {L}ipschitz surfaces: characterization via null-spaces of the scalar potential},
  year    = {2021},
  pages   = {4096-4117},
  volume  = {53},
}

@article{atfeh10,
	Author = {Atfeh, B. and Baratchart, L. and Leblond, J. and Partington, J.R.},
	Journal = {Journal of Fourier Analysis and Applications},
	Pages = {177-203},
	Title = {Bounded extremal and {C}auchy-{L}aplace problems on the sphere and shell},
	Volume = {16},
	Year = {2010}}

@article{baratchartgerhards16,
	Author = {Baratchart, L. and Gerhards, C.},
	Journal = {SIAM Journal on Applied Mathematics},
	Pages = {1756-1780},
	Title = {On the Recovery of Crustal and Core Contributions in Geomagnetic Potential Fields},
	Volume = {77},
	Year = {2017}}

@article{baratchart13,
	Author = {Baratchart, L. and Hardin, D.P. and Lima, E.A. and Saff, E.B. and Weiss, B.P.},
	Journal = {Inverse Problems},
	Pages = {015004},
	Title = {Characterizing Kernels of Operators Related to Thin Plate Magnetizations via Generalizations of {H}odge Decompositions},
	Volume = {29},
	Year = {2013}}

@book{freedengerhards12,
	Author = {Freeden, W. and Gerhards, C.},
	Publisher = {Chapman \& Hall/CRC},
	Series = {Pure and Applied Mathematics},
	Title = {Geomathematically Oriented Potential Theory},
	Year = {2012}}

@book{freedenschreiner09,
	Author = {Freeden, W. and Schreiner, M.},
	Publisher = {Springer},
	Title = {Spherical Functions of Mathematical Geosciences},
	Year = {2009}}

@incollection{gerhards18a,
	Author = {Gerhards, C.},
	Booktitle = {Handbook of Mathematical Geodesy},
	Editor = {Freeden, W. and Nashed, M.Z.},
	Publisher = {Springer},
	Title = {Spherical Potential Theory: Tools and Applications},
	Year = {2018}}

@article{gerhards16a,
	Author = {Gerhards, C.},
	Journal = {Inverse Problems},
	Pages = {015002},
	Title = {On the Unique Reconstruction of Induced Spherical Magnetizations},
	Volume = {32},
	Year = {2016}}

@article{gubbins11,
	Author = {Gubbins, D. and Ivers, D. and Masterton, S.M. and Winch, D.E.},
	Journal = {Geophysical Journal International},
	Pages = {99-117},
	Title = {Analysis of lithospheric magnetization in vector spherical harmonics},
	Volume = {187},
	Year = {2011}}

@article{lesur10,
	Author = {Lesur, V. and Wardinski, I. and Hamoudi, M. and Rother, M.},
	Journal = {Earth, Planets and Space},
	Pages = {765-773},
	Title = {The second generation of the {GFZ} {Reference} {Internal} {Magnetic} {Model}: {GRIMM-2}},
	Volume = {62},
	Year = {2010}}

@article{olsen10b,
	Author = {Olsen, N. and Glassmeier, K-H. and Jia, X.},
	Journal = {Space Science Reviews},
	Pages = {135-157},
	Title = {Separation of the Magnetic Field into External and Internal Parts},
	Volume = {152},
	Year = {2010}}

@incollection{olsen15,
	Author = {Olsen, N. and Hulot, G. and Sabaka, T.J.},
	Booktitle = {Handbook of Geomathematics},
	Edition = {2nd},
	Editor = {Freeden, W. and Nashed, M.Z. and Sonar, T.},
	Publisher = {Springer},
	Title = {Sources of the Geomagnetic Field and the Modern Data that Enable their Investigation},
	Year = {2015}}

@incollection{sabaka15,
	Author = {Sabaka, T.J. and Hulot, G. and Olsen, N.},
	Booktitle = {Handbook of Geomathematics},
	Edition = {2nd},
	Editor = {Freeden, W. and Nashed, M.Z. and Sonar, T.},
	Publisher = {Springer},
	Title = {Mathematical Properties Relevant to Geomagnetic Field Modeling},
	Year = {2015}}

@InCollection{plattnersimons14b,
  author    = {Plattner, A. and Simons, F. J.},
  booktitle = {Handbook of Geomathematics},
  publisher = {Springer},
  title     = {Potential field estimation using scalar and vector {S}lepian functions at satellite altitude},
  year      = {2015},
  edition   = {2nd},
  editor    = {Freeden, W. and Nashed, M.Z. and Sonar, T},
}

@article{verles18,
	Author = {Vervelidou, F. and Lesur, V.},
	Journal = {Geophysical Research Letters},
	Pages = {283-292},
	Title = {Unveiling {E}arth's Hidden Magnetization},
	Volume = {45},
	Year = {2018}}

@article{verles19,
	Author = {Lesur, V. and Vervelidou, F.},
	Doi = {doi.org/10.1093/gji/ggz471},
	Journal = {Geophysical Journal International},
	Pages = {981-995},
	Title = {Retrieving lithospheric magnetization distribution from magnetic field models},
	Volume = {220},
	Year = {2020}}

@Article{barvilhar19,
  author     = {Baratchart, L. and Villalobos Guill\'{e}n, C. and Hardin, D. P. and Northington, M. C. and Saff, E. B.},
  journal    = {Found. Comp. Math.},
  title      = {Inverse Potential Problems for Divergence of Measures with Total Variation Regularization},
  year       = {2020},
  pages      = {1273-1307},
  volume     = {20},
  doi        = {10.1007/s10208-019-09443-x},
}

@Article{olsen20,
  author  = {Finlay, C. C. and Kloss, C. and Olsen, N. and others},
  journal = {Earth, Planets and Space},
  title   = {The {CHAOS-7} geomagnetic field model and observed changes in the {S}outh {A}tlantic {A}nomaly},
  year    = {2020},
  pages   = {176},
  volume  = {72},
}

@Article{sabaka20,
  author  = {Sabaka, T. J. and T{\o}ffner-Clausen, L. and Olsen, N. and Finlay, C. C.},
  journal = {Earth, Planets and Space},
  title   = {{CM6}: a comprehensive geomagnetic field model derived from both {CHAMP} and {Swarm} satellite observations},
  year    = {2020},
  pages   = {80},
  volume  = {72},
}

@Article{isakov12,
  author  = {Elcrat, A. and Isakov, V. and Kropf, E. and Stewart, D.},
  journal = {Inverse Problems},
  title   = {A stability analysis of the harmonic continuation},
  year    = {2012},
  pages   = {075016},
  volume  = {28},
}

@Article{gerhuakeg23,
  author  = {Gerhards, C. and Huang, X. and Kegeles, A.},
  journal = {Journal of Mathematical Analysis and Applications},
  title   = {Relation between {H}ardy components for locally supported vector fields on the sphere},
  year    = {2023},
  pages   = {126572},
  volume  = {517},
}

@Article{Plattner2017,
  author  = {Plattner, A. and Simons, F. J.},
  date    = {2017},
  title   = {{Internal and external potential-field estimation from regional vector data at varying satellite altitude}},
  doi     = {10.1093/gji/ggx244},
  eprint  = {https://academic.oup.com/gji/article-pdf/211/1/207/19523195/ggx244.pdf},
  issn    = {0956-540X},
  pages   = {207-238},
  url     = {https://doi.org/10.1093/gji/ggx244},
  volume  = {211},
  journal = {Geophysical Journal International},
  year    = {2017},
}

@Book{mueller96,
  author    = {M\"uller, C.},
  title     = {Spherical {H}armonics},
  year      = {1996},
  publisher = {Springer},
  series    = {Lecture Notes in Mathematics},
}

@article{Gerhards2025,
  title = {Spherical basis functions in Hardy spaces with localization constraints},
  volume = {306},
  ISSN = {0021-9045},
  url = {http://dx.doi.org/10.1016/j.jat.2024.106124},
  DOI = {10.1016/j.jat.2024.106124},
  journal = {Journal of Approximation Theory},
  publisher = {Elsevier BV},
  author = {Gerhards,  C. and Huang,  X.},
  year = {2025},
  month = mar,
  pages = {106124}
}

@Article{Mayer2006,
  author       = {Mayer, C. and Maier, T.},
  date         = {2006},
  journaltitle = {Geophys. J. Int.},
  title        = {Separating Inner and Outer {E}arth's Magnetic Field from {CHAMP} Satellite Measurements by Means of Vector Scaling Functions and Wavelets},
  pages        = {1188--1203},
  volume       = {167},
  journal      = {Geophys. J. Int.},
  year         = {2006},
}

@Article{emag2,
  author  = {Maus, S. and Brackhausen, U. and Berkenbosch, H. and others},
  title   = {{EMAG2}: A 2–arc min resolution Earth Magnetic Anomaly Grid compiled from satellite, airborne, and marine magnetic measurements},
  volume  = {10},
  journal = {Geochem. Geophys. Geosys.},
  year    = {2009},
}

@Article{Hunt1968,
  author  = {Hunt, R. A. and Wheeden, R. L.},
  title   = {On the boundary values of harmonic functions},
  pages   = {307--322},
  volume  = {132},
  journal = {Trans. Am. Math. Soc.},
  year    = {1968},
}

@Article{plamazger24,
  author  = {Plattner, A. and Mazarico, E. and Gerhards, C.},
  title   = {Variable altitude cognizant {S}lepian functions},
  pages   = {16},
  volume  = {15},
  journal = {GEM - Int. J. Geomath.},
  year    = {2024},
}

@Article{aleron09,
  author  = {Alessandrini, G. and Rondi, L. and Rosset, E. and Vessella, S.},
  title   = {The stability for the {C}auchy problem for elliptic equations},
  pages   = {123004},
  volume  = {25},
  journal = {Inverse Problems},
  year    = {2009},
}

@Article{ruesal19,
  author  = {R\"uland, A. and Salo, M.},
  title   = {Quantitative {R}unge Approximation and Inverse Problems},
  pages   = {6216–6234},
  volume  = {20},
  journal = {Int. Math. Res. Not.},
  year    = {2019},
}

@Article{bou10,
  author  = {Bourgeois, L.},
  title   = {About stability and regularization of ill-posed elliptic {C}auchy problems: the case of $C^ {1,1}$ domains},
  pages   = {715-735},
  volume  = {44},
  journal = {ESAIM - Math. Model. Numer. Anal.},
  year    = {2010},
}

@Article{ren25,
  author  = {Ren, Z. and Zuo, Z. and Yao, H. and Chen, C. and Xu, L. and Tang, J. and Zhang, K.},
  title   = {Mag{TF}s: A tool for estimating multiple magnetic transfer functions to constrain {E}arth’s electrical conductivity structure},
  pages   = {105769},
  volume  = {195},
  journal = {Computers and Geosciences},
  year    = {2025},
}

@Article{kuv12,
  author  = {Kuvshinov, A.},
  title   = {Deep electromagnetic studies from land, sea, and space: progress status in the past 10 years},
  pages   = {169-209},
  volume  = {33},
  journal = {Surv. Geophys.},
  year    = {2012},
}

@Article{gray24,
  author  = {Grayver, A.},
  title   = {Unravelling the Electrical Conductivity of {E}arth and Planets},
  pages   = {187-238},
  volume  = {45},
  journal = {Surv. Geophys.},
  year    = {2024},
}

@Article{kuv15,
  author  = {P\"uthe, C. and Kuvshinov, A. and Olsen, N.},
  title   = {Handling complex source structures in global {EM} induction studies: from {C}-responses to new arrays of transfer functions},
  pages   = {318-328},
  volume  = {201},
  journal = {Geophys. J. Int.},
  year    = {2015},
}
\bibliographystyle{plain}

\normalsize
\appendix

\section{Divergence-free contributions}\label{app:divfree}

Beside the Hardy-Hodge decomposition $L^2(\SS)^3=\mathcal H_+(\SS)\oplus \mathcal H_-(\SS)\oplus \mathcal H_{\mathrm{df}}(\SS)$, there also exists the classical Helmholtz-Hodge decomposition $L^2(\SS)^3=\mathcal H_\eta(\SS)\oplus \mathcal H_{\mathrm{cf}}(\SS)\oplus \mathcal H_{\mathrm{df}}(\SS)$ of normal, tangential curl-free, and tangential divergence-free contributions (see, e.g., \cite{freedengerhards12,gerhards18a} and references therein). That is, any vector field $\f\in L^2(\SS)^3$ can be decomposed as $\f=\f_++\f_-+\f_{\mathrm{df}}=\eta f_\eta+\f_{\mathrm{cf}}+\f_{\mathrm{df}}=\eta f_\eta+\f_T$, with the divergence-free contribution $\f_{\mathrm{df}}$ appearing in both decompositions.

If we are interested in also including the divergence-free contribution in our previous study, we first observe that $f_\eta$ and $\f_T$ are uniquely determined on the patch $U$ by knowledge of $\f$ only on $U$. Since the divergence-free contribution can be expressed as $\f_{\mathrm{df}}=(\eta\times\nabla_\SS)f_{\mathrm{df}}$ for some $f_{\mathrm{df}}\in H^1(\SS)$, it holds that 
\[
\Delta_\SS\,f_{\mathrm{df}}=\mathrm{curl}_\SS\, \f_T
\]
in $H^{-1}(\SS)$. The prescription of adequate boundary conditions for $f_{\mathrm{df}}$ on $\partial U$ provides uniqueness of $\f_{\mathrm{df}}$ in $U$. Without such boundary conditions, no uniqueness of $\f_{\mathrm{df}}$ is guaranteed, even under the altitude/analyticity constraints of Section \ref{subsec:localized-conditioned}, if $\f$ is only known on $U$. However, given such boundary conditions, the computation of $\f_{\mathrm{df}}$ based on solving the Laplace-Beltrami equation is a standard elliptic boundary value problem and is stable in the corresponding well-posed boundary-value setting. The only potential source of instability comes through the application of $\mathrm{curl}_\SS$ to the tangential contribution $\f_T$ on the right hand side. This instability, however, is negligible compared to the exponential instability of the internal/external potential field separation.

\section{Nullspace characterization}\label{app:null}

\begin{corollary}\label{cor:nullspace}
    The nullspace of the operator $\mathcal A_U$ from Theorem \ref{thm:nonunique} takes the form
    \begin{align}
    	\ker(\mathcal A_U)=\big\{(B_+^{\ast} f,B_-^{\ast}g):\,&f= \chi_U(K^{\ast}+1/2)S^{-1}h_1+h_2,\, g=S^{-1}h_1-f,\label{eqn:nullspace}
	    	\\&\textnormal{for any }h_1\in \widetilde H^1_{\text{const}}(V),\, h_2\in L^2(V)\big\},\nonumber
	    \end{align} 
	    with $\chi_U$ denoting the characteristic function on $U$. 
\end{corollary}

\section{Auxiliary results}\label{app:aux}

\begin{proof}[Proof of \eqref{eqn:uminusest}]

    For an orthonormal basis of spherical harmonics $\{Y_{n,k}\}_{n\in\mathbb{N}_0,k=-n,\ldots,n}$, we have
	\[
		\int_{\SS}Y_{n,k}\overline{Y_{m,\ell}}\,d\sigma=\delta_{nm}\delta_{k\ell},
		\qquad
		\int_{\SS}\nabla_{\SS}Y_{n,k}\cdot\overline{\nabla_{\SS}Y_{m,\ell}}\,d\sigma
		=n(n+1)\delta_{nm}\delta_{k\ell}.
	\]
    With $\mathbf G_{n,k}^{\mathrm{int}}=\mathbf\Psi_{n,k}-(n+1)\mathbf Y_{n,k}$ and $\mathbf Y_{n,k}$ and $\mathbf\Psi_{n,k}=\nabla_{\SS}Y_{n,k}$ as in Section \ref{subsec:spherical-hhd}, this leads to
	\[
		\|\mathbf G_{n,k}^{\mathrm{int}}\|_{L^2(\SS)^3}^2
		= n(n+1)+(n+1)^2 =(n+1)(2n+1).
	\]
	Furthermore, if a function $u_-$ on the sphere in $\R^3\setminus \BB$ takes the form
	\[
		u_-(\rho,\xi)=\sum_{n=0}^\infty\sum_{k=-n}^n a_{n,k}\,q_n(\rho)Y_{n,k}(\xi),
		\qquad 1<\rho<1+s,\quad \xi\in\SS,
	\]
	with coefficients $a_{n,k}\in\R$, then, based on the expression $\nabla = \eta\,\partial_\rho + r^{-1}\nabla_\SS$, one can compute	\begin{align}
		\|u_-\|_{H^1(A_{1+s})}^2
		=
		\sum_{n=0}^\infty\sum_{k=-n}^n|a_{n,k}|^2
		\int_1^{1+s}
		\left[
			\rho^2|q_n(\rho)|^2
			+\rho^2|q_n'(\rho)|^2
			+n(n+1)|q_n(\rho)|^2
		\right]d\rho. \label{eqn:app-H1-shell-identity}
	\end{align}
    In our setup, we need to make the particular choice $q_n(\rho)=\rho^{-(n+1)}$ and, additionally, choose $s=\tfrac{r-1}{4}$. Then
	\[
		\|u_-\|_{H^1(A_{1+s})}^2
		=
		\sum_{n,k}|a_{n,k}|^2 I_n^-(r),
	\]
	with
	\begin{align*}
		I_n^-(r)
		&=
		\int_1^{1+s}
		\left[
			\rho^2|q_n^-(\rho)|^2
			+\rho^2|(q_n^-)'(\rho)|^2
			+n(n+1)|q_n^-(\rho)|^2
		\right]d\rho  \\
		&=
		\int_1^{1+s}
		\left[
			\rho^{-2n}
			+(n+1)^2\rho^{-2n-2}
			+n(n+1)\rho^{-2n-2}
		\right]d\rho
		\\&\le
		s\left[1+(n+1)^2+n(n+1)\right]
		\le
		C\,s\,(n+1)(2n+1).
	\end{align*}
	for some constant $C>0$ neither depending on $U$ nor $r$. The choice
	\[
		C_-(r)=\sup_{n\ge0}
		\sqrt{\frac{I_n^-(r)}{(n+1)(2n+1)}}
		<\infty
	\]
    finally provides us with the desired estimate \eqref{eqn:uminusest}, namely,
    \[
		\|u_-\|_{H^1(A_{1+s})}^2
		\le
		C_-^2(r)\|\f_-\|_{L^2(\SS)^3}^2\leq C_-^2(r)M^2,
	\]
    where we have chosen $\f_-=\nabla u_-|_{\SS}=\sum_{n=0}^\infty\sum_{k=-n}^na_{n,k}\mathbf G_{n,k}^{\mathrm{int}}$ and, consequently, $\|\f_-\|_{L^2(\SS)^3}^2=\sum_{n=0}^\infty\sum_{k=-n}^n|a_{n,k}|^2(n+1)(2n+1)$.

\end{proof}

\begin{lemma}[Weak boundary logarithmic stability]\label{lem:carleman}
Let $s=\tfrac{r-1}{4}$, so that $A_{1+s}\subset A_r$, and fix an arbitrary $\alpha\in(0,1)$. Then, for every $r>1$ and every a priori bound $\tilde{M}>0$, there exists a constant $C=C(r,U,\alpha)>0$ and a threshold $\delta=\delta(r,U,\alpha,\tilde{M})>0$ such that the following holds true: for all $w\in H^1(A_{r})$ that are harmonic in the annulus $A_r$, that satisfy
    \[
        \|w\|_{H^2(A_{1+s})}\le \tilde{M},
    \]
    and that obey
    \[
		0<\|w\|_{H^1(U)}+\|\partial_\eta^+ w\|_{L^2(U)}<\delta,
	\]
    we have
	\begin{align}
		\|\nabla w|_\SS\|_{H^{-1/2}(\SS)^3}
		\leq C \tilde{M}
		\left| \ln\left(\frac{c\tilde{M}}{\|w\|_{H^1(U)}+\|\partial_\eta^+ w\|_{L^2(U)}}\right)\right|^{-\alpha}.\label{eqn:carleman-new}
	\end{align}

\end{lemma}

\begin{proof}
For brevity, set
	\[
		\varepsilon=\|w\|_{H^1(U)}+\|\partial_\eta^+ w\|_{L^2(U)}.
	\]
	Applying the logarithmic Cauchy stability estimate from \cite[Cor.~2.1]{bou10} with the a priori bound $\|w\|_{H^2(A_{1+s})}\le \tilde{M}$, we obtain a constant $C_{\mathrm{cp}}(r,U,\alpha)>0$ and a threshold $\delta=\delta(r,U,\alpha,\tilde{M})>0$ such that, for all $0<\varepsilon<\delta$,
	\begin{align}
        \|w\|_{H^1(A_{1+s/2})}
		\le
		C_{\mathrm{cp}}(r,U,\alpha)\tilde{M}
		\left|
		\ln\left(
		\frac{\tilde{M}}{\varepsilon}
		\right)
		\right|^{-\alpha}. \label{eqn:carlest}
	\end{align}
	The weak boundary trace estimate for harmonic functions in the smooth spherical shell gives
	\[
		\|\nabla w|_\SS\|_{H^{-1/2}(\SS)^3}
		\le C_{\mathrm{tr}}(r)\|w\|_{H^1(A_{1+s/2})}.
	\]
	Indeed, the tangential component is controlled by the trace theorem:
	$w|_\SS\in H^{1/2}(\SS)$ and hence
	$\nabla_\SS \,w|_\SS\in H^{-1/2}(\SS)^3$. The normal component is understood
	as the weak Neumann trace, which is continuous for harmonic $H^1$-functions
	in the shell:
	\[
		\|\partial_\eta^+ w\|_{H^{-1/2}(\SS)}
		\le C_{\mathrm{tr}}(r)\|w\|_{H^1(A_{1+s/2})}.
	\]
	Consequently,
	\[
		\|\nabla w|_\SS\|_{H^{-1/2}(\SS)^3}
		\le C_{\mathrm{tr}}(r)\|w\|_{H^1(A_{1+s/2})}.
	\]
	Combining this weak trace estimate with \eqref{eqn:carlest}, and absorbing $C_{\mathrm{tr}}(r)C_{\mathrm{cp}}(r,U,\alpha)$ into $C(r,U,\alpha)$, yields \eqref{eqn:carleman-new}.

\end{proof}

\end{document}